\documentclass[reqno,11pt,a4paper]{amsart}
\usepackage[utf8]{inputenc}  
\usepackage[T1]{fontenc}  
\usepackage{amscd,amssymb,amsmath,latexsym,url,mathrsfs,upgreek,graphicx,mathtools}
\usepackage[all]{xy}
\usepackage{eufrak}
\usepackage{hyperref} \hypersetup{breaklinks=true}

\newcommand{\Res}{\operatorname{Res}}

\renewcommand{\and}{{\quad \text{ and } \quad}}

\makeatletter
\newcommand{\pushright}[1]{\ifmeasuring@#1\else\omit\hfill$\displaystyle#1$\fi}
\newcommand{\pushleft}[1]{\ifmeasuring@#1\else\omit$\displaystyle#1$\hfill\fi\ignorespaces}
\makeatother

\def \A{\mathbb{A}}

\def \C{\mathbb{C}}

\def \N{\mathbb{N}}

\def \Q{\mathbb{Q}}
\def \R{\mathbb{R}}

\def \Z{\mathbb{Z}}

\def\bfA{{\boldsymbol A}}
\def\bfB{{\boldsymbol B}}

\def\bfQ{{\boldsymbol Q}}

\newcommand{\bfa}{{\boldsymbol{a}}}
\newcommand{\bfb}{{\boldsymbol{b}}}

\newcommand{\bfd}{{\boldsymbol{d}}}

\newcommand{\bff}{{\boldsymbol{f}}}
\newcommand{\bfk}{{\boldsymbol{k}}}

\newcommand{\bfh}{{\boldsymbol{h}}}

\newcommand{\bfell}{{\boldsymbol{\ell}}}

\newcommand{\bfp}{{\boldsymbol{p}}}

\newcommand{\bfs}{{\boldsymbol{s}}}

\newcommand{\bfw}{{\boldsymbol{w}}}

\newcommand{\bfz}{{\boldsymbol{z}}}

\newcommand{\bfbeta}{{\boldsymbol{\beta}}}
\newcommand{\bflambda}{{\boldsymbol{\lambda}}}

\newcommand{\bfnu}{{\boldsymbol{\nu}}}

\newcommand{\bfzeta}{{\boldsymbol \zeta}}
\newcommand{\bfkappa}{{\boldsymbol \kappa}}
\newcommand{\bfeta}{{\boldsymbol \eta}}
\newcommand{\bfvarpi}{{\boldsymbol \varpi}}

\newcommand{\bfzero}{\boldsymbol{0}}
\newcommand{\bfone}{\boldsymbol{1}}

\numberwithin{equation}{section}
\theoremstyle{definition}
\newtheorem{defn}{Definition}
\numberwithin{defn}{section}

\newtheorem{rem}[defn]{Remark}
\newtheorem{exmpl}[defn]{Example}
\newtheorem{problem}[defn]{Problem}

\theoremstyle{plain}
\newtheorem{lem}[defn]{Lemma}

\newtheorem{prop}[defn]{Proposition}
\newtheorem{thm}[defn]{Theorem}

\newtheorem{prop-def}[defn]{Proposition-Definition}

\begin{document}

\title[About a non-standard interpolation problem]{About a non-standard interpolation problem}
\author[D. Alpay]{Daniel Alpay}
\address{Schmid College of Science and Technology\\
Chapman University\\
One University Drive
Orange, California 92866\\
USA}
\email{alpay@chapman.edu}
\urladdr{\url{http://www1.chapman.edu/~alpay}}

\author[Yger]{Alain Yger}
\address{ Institut de Mathématiques, Université de Bordeaux.
351 cours de la Libération, 33405 Talence, France}
\email{Alain.Yger@math.u-bordeaux.fr}
\urladdr{\url{http://www.math.u-bordeaux.fr/~ayger}}

\date{\today} 
\thanks{The authors thank the Foster G. and Mary McGaw Professorship in 
Mathematical Sciences, which supported this research}

\begin{abstract}
Using algebraic methods, and motivated by the one variable case, we study
a multipoint interpolation problem in the setting of several complex variables.
The  duality realized by the residue generator associated with an underlying Gorenstein algebra, using the Lagrange interpolation polynomial,
plays a key role in the arguments.
\end{abstract}

\maketitle

\overfullrule=0.3mm

\vspace*{-8mm}
\tableofcontents

\vspace*{-8mm}

\section[Introduction]{Introduction}
In \cite{AJLV15} the following multipoint interpolation problem was considered: 

\begin{problem}
Given complex numbers $a_{j,k}$ ($j=1,\ldots, m$ and $k=0,\ldots, \mu_j-1$) and $c$, 
describe the set of all functions $f$ analytic in a neighborhood $\Omega$ of the points $w_1,\ldots, w_m$ and such that
\begin{equation}
\label{pb00001}
\sum_{j=1}^m\sum_{k=0}^{\mu_j-1}a_{j,k}f^{(k)}(w_j)=c
\end{equation}
\end{problem}

Note that if $f$ solves \eqref{pb00001} so does $f+ph$, where
\begin{equation}
\label{repre1}
p(z)=\prod_{j=1}^n(z-w_j)^{\mu_j},\qquad N=\sum_{j=0}^m\mu_j.
\end{equation}
In other words one can work in the ideal $H(\Omega)/(p)$. In \cite{AJLV15} one used a different approach and
a key tool to solve the above problem was to represent any function analytic in $\Omega$ in the form
\begin{equation}
\label{pb001}
f(z)=\sum_{\nu=0}^{N-1}z^\nu f_\nu(p(z)),
\end{equation}
and $f_0,\ldots, f_{N-1}$ are analytic in a neighborhood of the origin. This representation allows to reduce condition \eqref{pb00001} to a tangential interpolation
condition at the origin for the $\mathbb C^N$-valued function $F=\begin{bmatrix}f_0&\cdots&f_{N-1}\end{bmatrix}^t$.\\

In the present paper we study the counterpart of the previous interpolation problem in the setting of several complex variables, see Problems 
\ref{problem1} and \ref{problem2} below. We now replace the polynomial \eqref{repre1} by a zero-dimensional ideal in $\mathbb C[\bfs]$ generated by
$n$ polynomials $p_1,\ldots, p_n$, and characterize the elements in the corresponding quotient space in terms of a duality realized by the residue generator associated 
with the Gorenstein algebra $\mathbb C[\mathbf s]/(p)$, using the Lagrange interpolation polynomial (see \eqref{lag123}
for the latter). This allows to define local coordinates and then translate the interpolation condition into an hyperplan condition in terms of these coordinates. Thus both in the one variable approach of \cite{AJLV15} and in the present work one 
reduces condition \eqref{pb00001} to a single interpolation condition.


\section[Zero dimensional polynomial ideals and afferent currents] {Zero dimensional polynomials ideals in $\C[s_1,...,s_n]$ and duality}
\label{section1} 

{\bf Notations.} In the polynomial algebra $\C[\bfs]$ ($\bfs=(s_1,...,s_n)$), one will denote, for any $\bfbeta \in \N^n$ as $\bfs^\bfbeta$ the monomial $s_1^{\beta_1}\dots 
s_n^{\beta_n}$. Given two elements $\bfell$ and $\bfell'$ in $\N^n$, by 
$\bfell \prec \bfell'$, we mean $\ell_j\leq \ell'_j$ for any 
$j=1,...,n$. We also denote $|\bfell|:= \ell_1 + \dots + \ell_n$ and 
$\bfell! := \ell_1!\cdots \ell_n!$.  
\vskip 2mm
  
A polynomial ideal $(\bfp) = (p_1,...,p_M)$ in $\C[\bfs]$ 
is said to be {\it zero-dimensional} if its zero set $\bfp^{-1}(0)=\{\bfzeta \in \C^n\,;\, 
p_1(\bfzeta)=\cdots = p_M(\bfzeta)=0\}$ is non-empty and discrete, hence finite since it is an algebraic subvariety in the affine space $\A^n_\C$. When additionally the number of polynomial generators equals the dimension, that is $M=n$, the set of generators $(p_1,...,p_n)$ is said to define a discrete complete intersection in $\C^n$ (or, equivalently, the sequence $(p_1,...,p_n)$ is a {\it quasi-regular sequence} in $\C[\bfs]$). 
\vskip 2mm

It is equivalent to say that $(\bfp)=(p_1,...,p_M)$ is zero-dimensional and that the 
$\C$-vector space $\C[\bfs]/(\bfp)$ is finitely dimensional, with 
$$
\dim_\C (\C[\bfs]/(\bfp)) = N(\bfp) \leq d_1d_2\cdots d_n
$$ 
(provided $d_2=\deg p_2 \geq d_3 =\deg p_3 \geq \cdots 
\geq d_M=\deg p_M\geq d_1 =\deg p_1$)~; this follows from B\'ezout geometric theorem.
In order to construct a monomial basis ${\mathscr B}_{\C[\bfs]/(\bfp)}^{\prec\prec} = \{\dot\bfs^{\bfbeta_k}\,;\, k=0,...,N(\bfp)-1\}$ for $\C[\bfs]/(\bfp)$, which will be required in order to settle the results presented in this paper, one proceeds algorithmically as follows~: 
\begin{itemize} 
\item decide of an order $\prec\prec$ on $\N^n$ (e.g the reverse lexicographic order, which is the most currently used)~; 
\item compute a Gr\"obner basis $\EuFrak G_{\bfp} := \{\EuFrak g_1,...,\EuFrak g_L\}_{\prec\prec}$ (with respect to the order $\prec\prec$ fixed from the beginning)~;  
\item collect all monomials that do not belong to the monomial ideal generated by the leading monomials (comparing their multi-exponents in terms of the order $\prec\prec$) of the polynomial entries in $\EuFrak G_{\bfp}$.
\end{itemize}

\begin{exmpl}\label{example1} In the particular case where $M=n$ and each $p_j$ is a 
univariate polynomial with degree $d_j$ in the single variable $s_j$, a monomial basis $\mathscr B_{\C[\bfs]/(\bfp)}^\bfd$ is provided (thanks to the Euclidean division algorithm with respect successively to the variables $s_1,...,s_n$) as 
\begin{equation}\label{basisfonda} 
\mathscr B_{\C[\bfs]/(\bfp)}^{\bfd} = \mathscr B_\bfp^{\rm euclid} := 
\{\dot \bfs^{\bfbeta}\,;\, \bfbeta \in \N^n\ {\rm with}\ \bfbeta \prec \bfd-\bfone\}\quad \big(\bfd:=(d_1,...,d_n)\big).   
\end{equation}    
\end{exmpl} 
\vskip 2mm

Zero-dimensional ideals in $\C[\bfs]$ that will be of interest for us in this paper will be generated by exactly 
$n$ polynomials $(p_1,...,p_n)$ (defining then a quasi-regular sequence in $\C[\bfs]$). In such a case, one can find 
a matrix $\bfA=[a_{j,k}]\in \mathscr M_{n,n}(\C[\bfs])$ and $n$ univariate polynomials $q_1(s_1),...,q_n(s_n)$ such that 
\begin{equation}\label{TL0}  
\begin{bmatrix}
q_1(s_1) \\ 
\vdots 
\\
q_n(s_n) 
\end{bmatrix} = \bfA(\bfs)\cdot \begin{bmatrix} 
p_1(\bfs)\\ \vdots \\ p_n(\bfs) 
\end{bmatrix}  
\end{equation}
with $\max_{j,k} (\deg q_j, \deg p_j + \deg a_{j,k}) \leq d_1\cdots d_n$ 
(resp. $\leq d_1 \cdots d_n d_{n+1}$), see \cite{Jel}. 
\vskip 2mm
\noindent    
Given a zero-dimensional ideal $(\bfp)=(p_1,...,p_M)$, the finitely dimensional $\C$-vector space ${\rm Hom}_\C(\C[\bfs]/(\bfp),\C)$ inherits a structure of $\C[\bfs]/(\bfp)$-module, setting 
\begin{equation}\label{defmodule} 
\forall\, \dot h \in \C[\bfs]/(\bfp), \quad 
\forall\, \Phi\, \in {\rm Hom}_\C(\C[\bfs]/(\bfp),\C),\quad 
\dot h \cdot \Phi = \Phi \circ \bfh, 
\end{equation} 
where $\bfh$ denotes the element in ${\rm Hom}_\C(\C[\bfs]/(\bfp),\C[\bfs]/(\bfp))$ 
which is induced by the multiplication by $h$. 
\vskip 2mm

In the particular case where $M=n$ and 
$\bfp = (p_1,...,p_n)$ is a quasi-regular sequence, the 
$\C[\bfs]/(\bfp)$-module $\C[\bfs]/(\bfp)\cdot {\rm Hom}_\C(\C[\bfs]/(\bfp),\C)$ defined as in \eqref{defmodule} is generated by the single element 
\begin{equation}\label{TL1} 
\dot h \longmapsto \Res \begin{bmatrix} 
h(\bfs)\, ds_1 \wedge \cdots \wedge ds_n \\ p_1(s),...,p_n(s)\end{bmatrix} = 
\Res \begin{bmatrix} 
h(\bfs)\, \det [\bfA(s)]\, ds_1 \wedge \cdots \wedge ds_n \\ q_1(s_1),...,q_n(s_n)\end{bmatrix},  
\end{equation} 
where the univariate polynomials $q_j$ and the matrix $\bfA \in 
\mathscr M_{n,n}(\C[\bfs])$ satisfy \eqref{TL0} (independently of the 
choice of the $q_j(s_j)$ and $\bfA$ such that such the matricial identity 
\eqref{TL0} holds). If $\bfd_q=(\deg q_1,...,\deg q_n)$ and $\bfone=(1,...,1)$, the 
right-hand side of \eqref{TL1} equals the coefficient $\tau_{\bfd_q-\bfone}$ of the monomial $s^{\bfd_q-\bf1}$ in the remainder $\sum_{\bfell \prec \bfd_q-\bfone} 
\tau_{\bfell} \bfs^\bfell$ after the successive euclidean divisions 
respectively by $q_1(s_1) = s_1^{\deg q_1}+\cdots$,...,$q_n(s_n)=s_n^{\deg q_n}+\cdots$
of the multivariate polynomial $h(\bfs)\det[\bfA(\bfs)]$ ($h$ being any representant of $\dot h$), see for example \cite{ElkMou}. The following important equivalence materializes in this case algebraic duality~: 
\begin{equation}\label{duality}  
\forall\, F \in \C[\bfs],\quad F \in \sum_{j=1}^n \C[\bfs]\, p_j \Longleftrightarrow 
\forall\, \Phi \in \C[\bfs],\ \Res \begin{bmatrix} 
\Phi(\bfs)\, F(\bfs)\, ds_1 \wedge \cdots \wedge ds_n \\ p_1(s),...,p_n(s)\end{bmatrix}=0,  
\end{equation}
which amounts to say that the quadratic form 
\begin{equation}\label{quadraticform} 
\bfQ_{\bfp}~: 
(\dot \Phi,\dot \Psi) \in (\C[\bfs]/(\bfp))^2 \longmapsto  \Res \begin{bmatrix} 
\Phi(\bfs)\, \Psi(\bfs)\, ds_1 \wedge \cdots \wedge ds_n \\ p_1(s),...,p_n(s)\end{bmatrix} \in \C
\end{equation}
is non-degenerated. The matrix of this non-degenerated quadratic form expressed in the monomial basis ${\mathscr B}_{\C[\bfs]/(\bfp)}^{\prec\prec} = \{\dot\bfs^{\bfbeta_k}\,;\, k=0,...,N(\bfp)-1\}$ for the $\C$-finite dimensional vector space 
$\C[\bfs]/(\bfp)$ is then 
\begin{multline}\label{matrixQp}  
\bfQ_{\bfp}[\mathscr B_{\C[\bfs]/(\bfp)}^{\prec\prec}] = 
\Bigg[ 
\Res \begin{bmatrix} 
s^{\bfbeta_{k_1} + \bfbeta_{k_2}}\, ds_1 \wedge \cdots \wedge ds_n \\ p_1(s),...,p_n(s)\end{bmatrix}\Bigg]_{0\leq k_1,k_2\leq N(\bfp)-1} \\ 
= \Bigg[ 
\Res \begin{bmatrix} 
s^{\bfbeta_{k_1} + \bfbeta_{k_2}}\, \det [\bfA(\bfs)]\, ds_1 \wedge \cdots \wedge ds_n \\ q_1(s_1),...,q_n(s_n)\end{bmatrix}\Bigg]_{0\leq k_1,k_2\leq N(\bfp)-1}. 
\end{multline} 
\vskip 2mm

When $M=n$, one can attach to the homomorphism \eqref{TL1} a unique complex valued $(0,n)$ current $\bigwedge_{j=1}^n \bar\partial (1/p_j)$ in $\C^n$ such that, whenever $U$ is an open subset of $\C^n$, $\overline h  \bigwedge_{j=1}^n (\bar\partial (1/p_j))_{|U} = 0$ as a $(0,n)$-current in $U$ 
for any $h\in H(U)$ which vanishes on $\bfp^{-1}(0)$ and moreover 
$$
\forall\, \dot h \in \C[\bfs]/(\bfp),\quad 
\Big\langle \bigwedge_{j=1}^n \bar\partial (1/p_j)\,,\, 
h(s)\, ds_1 \wedge \dots \wedge ds_n\Big\rangle = 
\Res \begin{bmatrix} 
h(\bfs)\, ds_1 \wedge \cdots \wedge ds_n \\ p_1(s),...,p_n(s)\end{bmatrix}.  
$$
Such a current can be defined in several ways. One of the most robust ones 
is the following (see \cite{BGVY},  
\cite{Sam})~: for any $(n,0)$-test form $\varphi(\bfs)\, ds_1\wedge \cdots \wedge ds_n$ where $\varphi \in \mathscr D(\C^n,\C)$ is compactly supported sufficiently close from $\bfp^{-1}(0)$, the holomorphic mapping 
\begin{multline*}
\bfvarpi = (\varpi_1,...,\varpi_n) 
\in \{\bfvarpi
\in \C^n\,;\, {\rm Re}\, \varpi_j > 1\ {\rm for}\ j=1,...,n\} 
\longmapsto \\
\frac{1}{(2i\pi)^n} 
\int_{\C^n} 
\bar\partial 
\Big(\frac{|p_n(\bfs)|^{2\varpi_n}}{p_n(\bfs)}\Big)  
\wedge \cdots \wedge \bar \partial \Big(\frac{|p_1(\bfs)|^{2\varpi_1}}{p_1(\bfs)}\Big)\wedge 
\varphi(\bfs) \, ds_1 \wedge \cdots \wedge ds_n  
\end{multline*} 
extends as an holomorphic function in 
$\{\bfvarpi \in \C^n\,;\, {\rm Re}\, \varpi_j > -\eta\}$ for some 
$\eta>0$, which value at $\bfvarpi=0$ equals precisely 
$\big\langle \bigwedge_{j=1}^n \bar\partial (1/p_j),\varphi\, ds_1 \wedge 
\cdots \wedge ds_n\rangle$.  

\begin{exmpl}\label{example2}  
If $p_1,...,p_n$ are $n$ univariate monic polynomials in the respective variables 
$s_1,...,s_n$ with 
\begin{equation}\label{univariate} 
p_j(\bfs) = p_j(s_j) = \prod\limits_{\kappa_j=1}^{m_j} (s_j-\xi_{j,\kappa_j})^{\nu_{j,\kappa_j}},\ j=1,...,n,  
\end{equation} 
one has 
\begin{multline}\label{example2form}  
\Big\langle \bigwedge_{j=1}^n \bar\partial (1/p_j)\,,\, 
\varphi(\bfs)\, ds_1 \wedge \dots \wedge ds_n\Big\rangle = 
\sum\limits_{\kappa_1=1}^{m_1} 
\cdots \sum\limits_{\kappa_n=1}^{m_n} 
\Big(\prod_{j=1}^n 
\frac{1}{(\nu_{j,\kappa_j}-1)!}\Big) \\
\Big( 
\frac{\partial^{\nu_{1,\kappa_1}-1}}
{\partial s_1^{\nu_{1,\kappa_1}-1}} \circ 
\cdots \circ \frac{\partial^{\nu_{n,\kappa_n}-1}}
{\partial s_n^{\nu_{n,\kappa_n}-1}}\Big)
\Big[\varphi (\bfs)\, 
\prod\limits_{j=1}^n \frac{(s_j- \xi_{j,\kappa_j})^{\nu_{j,\kappa_j}}}
{p_j(s_j)}\Big] (\xi_{1,\kappa_1},...,\xi_{n,\kappa_n}).   
\end{multline} 
\end{exmpl} 
\vskip 2mm 

The analytic pendant of the realization of algebraic duality \eqref{duality} is then ~: 
\begin{multline}\label{duality1} 
\forall\, U \ {\rm open\ subset\ of}\ 
\C^n,\quad \forall\, f\in H(U),\quad 
f \in \Big(\sum_{j=1}^n H(U)\, p_j\Big)_{{\rm loc}} \\ 
\Longleftrightarrow 
\forall\, \varphi \in \mathscr D(U,\C),\ 
\Big\langle \bigwedge_{j=1}^n \bar\partial (1/p_j)\,,\, f(\bfs)\, \varphi(\bfs)\, 
ds_1 \wedge \cdots \wedge ds_n \Big\rangle =0.  
\end{multline} 
\vskip 2mm 

In order to describe more precisely the current $\bigwedge_{j=1}^n 
\bar\partial (1/p_j)$ when $(p_1,...,p_n)$ is a quasi-regular sequence in 
$\C[\bfs]$ (as we did in Example \ref{example2} in the particular case where each $p_j$ is univariate in the single variable $s_j$, see \eqref{example2form}), we need to recall how each of the distinct points $w_j$, $j=1,...,m$, of the set $\bfp^{-1}(0)$ is equipped 
with a {\it multiplicity} $\nu_{w_j}(\bfp)\in \N^*$. Given $w_j\in \bfp^{-1}(0)$, such an integer $\nu_{w_j}(\bfp)$ can be defined in two ways~: 
\begin{itemize} 
\item ``algebraically'', as the dimension of the $\C$-vector space 
$\mathscr O_{\C^n,w_j}/(\bfp)_{w_j}$, where $(\bfp)_{w_j}$ denotes 
the ideal generated by the germs at $w_j$ of the polynomials $p_1,...,p_n$ 
in the local regular ring $\mathscr O_{\C^n,w_j}$ of germs of holomorphic functions about the point $w_j$~; 
\item ``dynamically'', as the number of points in the fiber 
$\bfp^{-1}(\bfeta)$ which remain close to $w_j$ when $\bfeta\in (\C^*)^n$
tends to $\bfzero$ in $\C^n$ and is taken as a non-critical value for the polynomial map $\bfp$.  
\end{itemize} 
If one uses the first definition, it is easy to see that 
$$
N(\bfp) = \dim_\C \big(\C[\bfs]/(\bfp)\big) = \sum_{j=1}^m \nu_{w_j}(\bfp). 
$$    
It follows from \eqref{duality1}, together with the fact that the 
N\oe ther exponent of the ideal $(\bfp)_{w_j}$ in 
$\mathscr O_{\C^n,w_j}$ is bounded from above by 
$n\, \nu_j(\bfp)$ (see \cite{Pl,Tsi}), that the order 
of the current $\bigwedge_{j=1}^n \bar\partial(1/f_j)$ about 
the point $w_j$ is at most $n\, \nu_{w_j}(\bfp)-1$. Therefore there 
exists a collection of differential operators $\mathscr Q_{w_1}(\partial/\partial \bfs),...,\mathscr Q_{w_m}(\partial/\partial \bfs)\in \C[\partial/\partial \bfs]$ (where $\partial/\partial \bfs = 
(\partial/\partial s_1,...,\partial/\partial s_n)$) such that 
$\deg_{\partial/\partial \bfs} \mathscr Q_{w_j} \leq n \nu_{w_j}(\bfp)-1$ for $j=1,...,m$ and 
\begin{multline}\label{representationresidualcurrent}  
\Big\langle \bigwedge_{j=1}^n \bar\partial (1/p_j)\,,\, 
\varphi(\bfs)\, ds_1 \wedge \dots \wedge ds_n\Big\rangle = 
\sum\limits_{j=1}^n \mathscr Q_{w_j}(\partial/\partial \bfs)[\varphi](w_j)\\ 
\forall\, \varphi \in \mathscr D(\C^n,\C).  
\end{multline} 

\begin{exmpl}\label{example3} If $p_1,...,p_n$ are monic univariate polynomials respectively in the variables $s_1,...,s_n$ as in Example \ref{example2} (more precisely of the form 
\eqref{univariate}), the multiplicity 
$\nu_w(\bfp)$ at the point $w=(\xi_{1,\kappa_1},...,\xi_{n,\kappa_n})$ equals 
$\prod_{j=1}^n \nu_{j,\kappa_j}$ and the order of the differential operator 
$\mathscr Q_w(\partial/\partial \bfs)$ attached to such $w\in \bfp^{-1}(0)$ 
as in \eqref{example2form} equals in this case $\prod_{j=1}^n (\nu_{j,\kappa_j}-1)$, which happens to be strictly less than $n\, \nu_{w}(\bfp)-1$ which should stand for the estimate of $\deg \mathscr Q_w$ in the general case. In the general case estimates 
$\deg \mathscr Q_{w_j}\leq n\, \nu_{w_j}(\bfp)-1$ ($j=1,...,m$) cannot in fact be sharpened. 
\end{exmpl} 

\vskip 2mm 

When $U$ is an open subset of $\C^n$, $f\in H(U)$ and 
$\varphi\in \mathscr D(U,\C)$, it follows from the Leibniz rule, together with the symetry of the left-hand side expression in $(\varphi,f)$ from the computational point of view, that one can write 
\begin{equation}\label{reformulation} 
\begin{split}  
& \Big\langle \bigwedge_{j=1}^n \bar\partial (1/p_j)\,,\, 
f(\bfs)\, \varphi(\bfs)\, ds_1 \wedge \dots \wedge ds_n\Big\rangle \\
& = 
\sum\limits_{\{j\in \{1,...,m\}\,;\, w_j\in U\}}\;   
\sum\limits_{\bfell \in A_{w_j}}
\frac{\partial^{|\bfell|}}{\partial \bfs^\bfell}[\varphi](w_j)\, 
\mathscr Q_{w_j,\bfell}(\partial/\partial \bfs)[f](w_j)\\ 
& = \sum\limits_{\{j\in \{1,...,m\}\,;\, w_j\in U\}}\  
\sum\limits_{\bfell \in A_{w_j}}
\frac{\partial^{|\bfell|}}{\partial \bfs^\bfell}[f](w_j)\, 
\mathscr Q_{w_j,\bfell}(\partial/\partial \bfs)[\varphi](w_j), 
\end{split} 
\end{equation}  
where $A_{w_1},...,A_{w_n}$ are finite subsets of $\N^n$ which are such 
that 
$$
\bfell \in A_{w_j} \Longrightarrow |\bfell|\leq n\, \nu_{w_j}(\bfp)-1\quad  
\forall\, j=1,...,m
$$
and each $\mathscr Q_{w_j,\bfell}$ ($j=1,...,m$, $\bfell\in A_{w_j}$) 
is a polynomial in $\partial/\partial \bfs$ with total degree at most 
$n\, \nu_{w_j}(\bfp)-1-|\bfell|$ with support in $A_{w_j}$. Note that the $A_{w_j}$ and the 
$\mathscr Q_{w_j,\bfell}$ ($j=1,...,m$, $\bfell\in A_{w_j}$) 
depend only on the differential operators $\mathscr Q_{w_j}(\partial/\partial \bfs)$ 
involved in \eqref{representationresidualcurrent}, hence only on the given polynomial quasi-regular sequence $\bfp=(p_1,...,p_n)$.  

\begin{defn}\label{noetherian}  
The list of differential operators with complex coefficients (and assigned evaluations) 
\begin{equation}\label{listopnoetheriens} 
{\rm Noeth}_\bfp[(\bfp)] := \Big\{\mathscr Q_{w_j,\bfell}(\partial/\partial \bfs)_{|w_j}\,;\, 
j=1,...,m,\ \bfell \in A_{w_j}\Big\} 
\end{equation} 
will be called {\it the standard list of assigned N\oe therian differential operators for the ideal $(\bfp)$ when considered as generated by the quasi-regular sequence $\bfp=(p_1,...,p_n)$}. 
\end{defn} 

\vskip 2mm

Thanks to Definition \ref{noetherian}, one can reformulate \eqref{duality1} as 
\begin{multline}\label{duality2} 
\forall\, U \ {\rm open\ subset\ of}\ 
\C^n,\quad \forall\, f\in H(U),\quad 
f \in \Big(\sum_{j=1}^n H(U)\, p_j\Big)_{{\rm loc}} \\ 
\Longleftrightarrow 
\forall\, w_j \in U,\ \forall\, \bfell \in A_{w_j},\ 
\mathscr Q_{w_j,\bfell}(\partial/\partial \bfs)[f](w_j) =0. 
\end{multline}   

\begin{exmpl}\label{example4} When the polynomials $p_j$ are monic univariate polynomials 
respectively in the variables $s_1,...,s_n$ as \eqref{univariate}, one has 
\begin{multline}\label{noetherianexample}
{\rm Noeth}_{(p_1(s_1),...,p_n(s_n))}[(p_1(s_1),...,p_n(s_n))]= \\ 
\Big\{
\frac{\partial^{|\bfell|}}{\partial s^\bfell}_{|(\xi_{1,\kappa_1},...,\xi_{n,\kappa_n})}\,;\, 1\leq \kappa_j\leq m_j\ {\rm for}\ j=1,...,m,\ \bfell\prec 
(\nu_{1,\kappa_1},...,\nu_{n,\kappa_n})-\bfone\Big\} 
\end{multline} 
\end{exmpl}

\section{Cauchy-Weil's integral representation formula and Lagrange interpolation} 

Cauchy-Weil's integral representation formula (originally introduced in \cite{Weil})  plays a major role 
in this paper. Let us briefly recall it in the particular simple case where it happens to be the most useful (one refers for example to 
\cite{AY,BGVY,Glea,BoH, TsiY} for a more detailed as well as a presentation in its generality in the analytic or algebraic context). Let $f_1,...,f_n$ 
be $n$ holomorphic functions in a bounded open set $U \subset \C^n$ (possibly not connected) and continuous up to $\partial U$, with no common zero on $\partial U$,  
such that additionally there exists a matrix $\bfB_\bff\in \mathscr M_{n,n}\big(H\big(U\times U) \cap 
C(\overline U \times U)\big)$ with 
$$
\forall \, j\in \{1,...,n\},\quad 
\forall\, \bfs,\bfz\in U,\quad 
f_j(\bfs) - f_j(\bfz) = \sum\limits_{\ell =1}^n b_{j,\ell}(\bfs,\bfz)\, (s_j-z_j). 
$$   
Such a matrix $\bfB_\bff$ (which is definitevely non unique as soon as $n>1$) is called an {\it Hefer matrix} or a {\it B\'ezoutian matrix} when the $f_j$ happen to be (as it will be the case in this paper) polynomial functions. 
The set $V(\bff):=\{w\in \overline U\,;\, \bff(w)=0\}$ is necessarily finite since 
$\bff$ does not vanish on $\partial U$. 
For almost all 
$\varepsilon \in (\R^+)^n$ such that $\|\varepsilon\|$ is small enough, the so-called {\it Weil analytic polyhedron} 
$$
\Delta_{\varepsilon}~:= \{\bfs\in U\,;\, 
|f_j(\bfs)|<\varepsilon_j,\ j=1,...,n\}
$$
is relatively compact in $U$ 
and, provided $\varepsilon$ is not a critical value of 
the smooth map $\bfs\in U \mapsto (|f_1(\bfs)|^2,...,|f_n(\bfs)|^2)$ 
(the set of such critical values being negligible in 
$(\R^+)^n$ according to A. Sard's lemma), is such its 
{\it Shilov boundary} 
$$
\Gamma_{\rm Shilov}(\overline{\Delta_\varepsilon}) := 
\{\bfs\in \overline{\Delta_\varepsilon}\,;\, 
|f_j(\bfs)|=\varepsilon_j,\ j=1,...,n\}
$$
is a real analytic $n$-dimensional manifold which will be oriented as follows~: the $n$-differential form $\bigwedge_{j=1}^n d[{\rm arg} (f_j(\bfs))]$ on will be a $n$-volume form on it. We denote then as 
$\Gamma_{\rm Shilov}^+(\overline{\Delta_\varepsilon})$ the corresponding real-analytic $n$-cycle. Then any holomorphic function $f\in H(U) \cap C(\overline U)$ can be represented in 
$\Delta_\varepsilon$ (whenever this Weil polyhedron is connected or not) as 
\begin{multline}\label{Weilformula} 
\forall\, \bfz\in U,\quad 
f(\bfz) = \\
\frac{1}{(2i\pi)^n} 
\int_{\Gamma_{\rm Shilov}^+(\overline{\Delta_\varepsilon})} 
f(\bfs)\, \det\, [\bfB_\bff(\bfs,\bfz)]\, \frac{ds_1 \wedge \dots \wedge ds_n} 
{\big(f_1(\bfs) - f_1(\bfz)\big) \cdots \big( f_n(\bfs) - f_n(\bfz)\big)}. 
\end{multline}

\vskip 2mm

In this section, one considers a zero-dimensional polynomial ideal generated by a quasi-regular sequence $\bfp=(p_1,...,p_n)$, which means that $\bfp^{-1}(0)$ is a non-empty finite set $\{w_1,...,w_n\}$ in $\C^n$ (supposed distinct, each of them been equipped with a multiplicity $\nu_{w_j}(\bfp)$ such that $\sum_{j=1}^m \nu_{w_j}(\bfp)= \dim_\C (\C[\bfs]/(\bfp))=N(\bfp)$). In all this section, let us also suppose that a monomial basis 
$$
{\mathscr B}_{\C[\bfs]/(\bfp)}^{\prec\prec} = \{\dot\bfs^{\bfbeta_k}\,;\, k=0,...,N(\bfp)-1\}
$$  
for $\C[\bfs]/(\bfp)$ has been obtained thanks to the search for 
a Gr\"obner basis for $(\bfp)$ with respect to the prescribed ordering 
$\prec\prec$ on monomials in $\C[\bfs]$ (as recalled in section \ref{section1}).  

\begin{prop}\label{propWeil1} Let $\bfp=(p_1,...,p_n)$ and 
${\mathscr B}_{\C[\bfs]/(\bfp)}^{\prec\prec}$ as above. 
Let $U$ be an open subset of $\C^n$ which contains $\bfp^{-1}(0)$ and $f\in H(U)$. 
There is a unique system of coordinates $(\alpha_0[f],...,\alpha_{N(\bfp)-1}[f])
\in \C^n$ such that the holomorphic function 
$$
\bfs\in U \longmapsto f(\bfs) - \sum\limits_{k=0}^{N(\bfp)-1} 
\alpha_k[f]\, \bfs^{\bfbeta_k}    
$$
belongs the ideal $(\sum_{j=1}^n H(U) p_j)_{\rm loc}$. 
\end{prop} 

\begin{proof}The proof of the unicity clause goes as follows~: if a polynomial function 
which restriction to $U$ is $\bfs \in U \mapsto \sum_{k=0}^{N(\bfp)-1} 
(\alpha_k - \tilde \alpha_k) \bfs^{\beta_k}$ 
belongs  to $\big(\sum_{j=1}^n H(U)\, p_j\big)_{\rm loc}$, it implies since 
$\bfp^{-1}(0)\subset U$ that the polynomial $\sum_{k=0}^{N(\bfp)-1} 
(\alpha_k - \tilde \alpha_k) \bfs^{\bfbeta_k}$ belongs to 
$\C[\bfs] p_1 + \cdots + \C[\bfs]\, p_n$, that is $\alpha_k=\tilde \alpha_k$ for 
$k=0,...,N(\bfp)-1$ since the collection $\{\dot \bfbeta^k\,;\, k=0,...,N(\bfp)-1\}$ is a basis of the quotient $\C$-vector space 
$\C[\bfs]/(\bfp)$. 
\\
\noindent 
As for the existence, one proceeds as follows. 
Let $\bfB_\bfp\in \mathscr M_{n,n}(\C[\bfs,\bfz])$ be any B\'ezoutian matrix of 
polynomials in $2n$ variables $(\bfs,\bfz)=(s_1,...,s_n,z_1,...,z_n)$ such that the following polynomial identities hold in $\C[\bfs,\bfz]$~: 
\begin{equation}\label{bezoutian}
p_j(\bfs)-p_j(\bfz) = \sum\limits_{\ell=1}^n \bfb_{j,\ell}(\bfs,\bfz)\, (s_\ell - z_\ell),\quad j=1,...,n.    
\end{equation}
Such a matrix $\bfB_\bfp$ always exists~: one can for example either invoke the so-called Fundamental Theorem of Analysis and take
$$
\bfb_{\bfp,j,\ell}(\bfs,\bfz) := \int_0^1 \Big(\frac{\partial}{\partial z_\ell}\Big)[p_j\big(\bfz + t(\bfs-\bfz)\big)]\, dt\quad \forall\, j,\ell \in \{1,...,n\}
$$   
or better proceed iteratively as follows for each $j=1,...,n$~: 
\begin{multline*} 
p_{j}(\bfs)-p_j(\bfz) = \frac{p_j(s_1,s_2,...,s_n)-p_j(z_1,s_2,...,s_n)}{s_1-z_1}\, (s_1-z_1) \\ 
+ \frac{p_j(z_1,s_2,s_3,...,s_n) - p_j(z_1,z_2,s_3,...,s_n)}{s_2-z_2}\, (s_2-z_2) + 
\cdots   
\end{multline*} 
in order to keep track of the smallest subring $\A\subset \C$ 
(for example $\A=\Z$ or $\A=\Q$) that contains all coefficients of $\bfp$ (if all 
$p_j$ lie in $\A[\bfs]$, so do then all entries of such $\bfB_\bfp$).   
Consider then a Weil polyhedron $\Delta\subset \subset U$ subordonned to $(p_1,...,p_n)$ in the open set $U$. Cauchy-Weil's integral representation formula \eqref{Weilformula}, together with the fact that the rational function $\zeta \mapsto 1/(1-\zeta)$ can be expanded normally on any compact of the unit disk $D(0,1)$ as $\sum_{\ell\geq 0} \zeta^\ell$, imply 
\begin{multline}\label{CW1}  
\forall\,z \in \Delta,\quad 
f(\bfz) =  \frac{1}{(2i\pi)^n} 
\int_{\Gamma_{\rm Shilov}^+(\overline \Delta)}\, f(\bfs)\, \frac{\det [\bfB_\bfp(\bfs,\bfz)]\, ds_1 \wedge \cdots\wedge ds_n}{\prod_{j=1}^n \big(p_j(\bfs)-p_j(\bfz)\big)}\\ 
= \frac{1}{(2i\pi)^n} 
\int_{\Gamma_{\rm Shilov}^+(\overline \Delta)}\, f(\bfs)\, \frac{\det [\bfB_\bfp(\bfs,\bfz)]}{\prod\limits_{j=1}^n p_j(\bfs)}\, \Bigg(\prod\limits_{j=1}^n \, 
\frac{1} {1 - \displaystyle{\frac{p_j(\bfz)}{
p_j(\bfs)}}}\Bigg)\, ds_1 \wedge \cdots\wedge ds_n \\
= \frac{1}{(2i\pi)^n} 
\int_{\Gamma_{\rm Shilov}^+(\overline \Delta)}\, f(\bfs)\, \frac{\det [\bfB_\bfp(\bfs,\bfz)]\, ds_1 \wedge \cdots\wedge ds_n}{\prod_{j=1}^n p_j(\bfs)} \\ 
+ \sum\limits_{\bfell \in \N^n\setminus 
\bf0} 
\Bigg(\frac{1}{(2i\pi)^n} 
\int_{\Gamma_{\rm Shilov}^+(\overline \Delta)}\, f(\bfs)\, \frac{\det [\bfB_\bfp(\bfs,\bfz)]\, ds_1 \wedge \cdots\wedge ds_n}{\prod_{j=1}^n p_j^{\ell_j+1}(\bfs)} 
\Bigg)\, \bfp^\bfell(\bfz) \\ 
= \frac{1}{(2i\pi)^n} 
\int_{\Gamma_{\rm Shilov}^+(\overline \Delta)}\, f(\bfs)\, \frac{\det [\bfB_\bfp(\bfs,\bfz)]\, ds_1 \wedge \cdots\wedge ds_n}{\prod_{j=1}^n p_j(\bfs)} 
+ \sum\limits_{j=1}^n p_j(\bfz)\, g_{\Delta,j}(\bfz),        
\end{multline}
where $g_{\Delta,j}\in H(\Delta)$ for any $j=1,...,n$. If one takes as 
$(\alpha_0[f],...,\alpha_{N(\bfp)-1}[f])$ the vector of coordinates (in the basis 
${\mathscr B}_{\C[\bfz]/(\bfp(\bfz))}^{\prec\prec}$) 
of the class in $\C[\bfz]/(\bfp(\bfz))$ of the polynomial 
$$
\frac{1}{(2i\pi)^n} 
\int_{\Gamma_{\rm Shilov}^+(\overline \Delta)}\, f(\bfs)\, \frac{\det [\bfB_\bfp(\bfs,\bfz)]\, ds_1 \wedge \cdots\wedge ds_n}{\prod_{j=1}^n p_j(\bfs)} \in \C[\bfz],  
$$
one gets the required result. 

\begin{rem}\label{remProp1}  
The polynomial (considered here in $\C[\bfz]$)   
\begin{equation}
\label{lag123}
{\rm Lag}[f] := 
\frac{1}{(2i\pi)^n} 
\int_{\Gamma_{\rm Shilov}^+(\overline \Delta)}\, f(\bfs)\, \frac{\det [\bfB_\bfp(\bfs,\bfz)]\, ds_1 \wedge \cdots\wedge ds_n}{\prod_{j=1}^n p_j(\bfs)}   
\end{equation}
depends only on the list of germs $[f_{w_1},...,f_{w_n}]$ of germs  
of $f$ respectively about each of the distinct points $w_j$ in 
$\bfp^{-1}(0)$. It then can be considered as a {\it Lagrange interpolator} of such list of germs, hence the terminology used here to denote it.   
More precisely let 
$\varphi\in \mathscr D(U,[0,1])$ be any test-function which equals identically $1$ about each point $w_j$ for $j=1,...,m$. It is worth to point out 
that ${\rm Lag}[f]$ expresses alternatively (independently of the choice of the Weil polyhedron $\Delta$) as 
\begin{equation}\label{expressionresidu}
\begin{split}   
& {\rm Lag}[f]= 
\Big\langle 
\Big(\bigwedge\limits_{j=1}^n 
\bar\partial(1/f_j)\Big)(\bfs)\,,\, f(\bfs)\,
\det [\bfB_\bfp(\bfs,\bfz)]\, \varphi(\bfs)\, ds_1 \wedge \cdots \wedge ds_n \Big\rangle \\ 
& = \sum\limits_{j=1}^m \mathscr Q_{w_j} 
(\partial/\partial \bfs) \big[ f(\bfs) 
\det [\bfB_\bfp(\bfs,\bfz)]\big] (w_j) \\
& = \sum\limits_{j=1}^m\;   
\sum\limits_{\bfell \in A_{w_j}}
\frac{\partial^{|\bfell|}}{\partial \bfs^\bfell}[f](w_j)\, 
\mathscr Q_{w_j,\bfell}(\partial/\partial \bfs)\big[\det [\bfB_\bfp(\bfs,\bfz)]\big](w_j)\\ 
& = \sum\limits_{j=1}^m\  
\sum\limits_{\bfell \in A_{w_j}}
\frac{\partial^{|\bfell|}}{\partial \bfs^\bfell}\big[\det [\bfB_\bfp(\bfs,\bfz)] \big](w_j)\, 
\mathscr Q_{w_j,\bfell}(\partial/\partial \bfs)[f](w_j), 
\end{split} 
\end{equation}
where the differential operators with complex coefficients 
$\mathscr Q_{w_j}(\partial/\partial \bfs)$ for $j=1,...,m$ 
(respectively the finite sets $A_{w_j} \in \N^n$ together with differential operators $\mathscr Q_{w_j,\bfell}$ for $j=1,...,m$ and $\bfell \in A_{w_j}$)  
are those introduced in 
\eqref{representationresidualcurrent} (respectively in \eqref{reformulation}). 
We refer here the reader for example  
to \cite{BGVY} or to the more up-to-date survey \cite{TsiY}.  
\end{rem}  
 
\end{proof} 

\begin{prop}\label{propWeil2} 
Let $\bfp$ and ${\mathscr B}_{\C[\bfs]/(\bfp)}^{\prec\prec}$ as above. 
Let $[h_{w_1},...,h_{w_m}]$ be a list of $m$ germs of holomorphic functions, each respectively about the zero $w_j$ of $\bfp$. There is a unique system of coordinates 
$\big(\alpha_0([\bfh_\bfw]),...,\alpha_{N(\bfp)-1}([\bfh_\bfw])\big)\in \C^n$ such that for each $j=1,...,m$, one has, as elements in the local ring $\mathscr O_{\C^n,w_j}$,  
$$
h_{w_j} - \Big(\bfs \mapsto \sum_{k=0}^{N(\bfp)-1} \alpha([\bfh_\bfw])\, \bfs^{\bfbeta_k}\Big)_{w_j} \in \mathscr O_{\C^n,w_j}\, p_{1,w_j}  + 
\cdots + \mathscr O_{\C^n,w_j}\, p_{n,w_j}.     
$$
\end{prop} 

\begin{proof} Take $U$ as a union of balls with infinitesimal small radii about each $w_j$ (so that, for any $j=1,...,m$, the germ $h_{w_j}\in \mathscr O_{\C^n,w_j}$  admits a representant $h_j$ in the connected component of $U$ that contains $w_j$). 
Take now the holomorphic function $f~: U \rightarrow \C$ defined by $f_{|B_j}=h_j$ for $j=1,...,m$ and then conclude appealing to Proposition \ref{propWeil1}. It also follows from Remark \ref{remProp1} that the system of coordinates 
$\big(\alpha_0([\bfh_\bfw]),...,\alpha_{N(\bfp)-1}([\bfh_\bfw])\big)$ 
is that of the class of the polynomial (considered here in $\C[\bfz]$) 
\begin{equation} 
\label{expressionresidu2} 
\begin{split}
& {\rm Lag}[h_{w_1},...,h_{w_m}] \\
& = \sum\limits_{j=1}^m\;   
\sum\limits_{\bfell \in A_{w_j}}
\frac{\partial^{|\bfell|}}{\partial \bfs^\bfell}[h_{w_j}](w_j)\, 
\mathscr Q_{w_j,\bfell}(\partial/\partial \bfs)\big[\det [\bfB_\bfp(\bfs,\bfz)]\big](w_j)\\ 
& = \sum\limits_{j=1}^m\  
\sum\limits_{\bfell \in A_{w_j}}
\frac{\partial^{|\bfell|}}{\partial \bfs^\bfell}\big[\det [\bfB_\bfp(\bfs,\bfz)] \big](w_j)\, 
\mathscr Q_{w_j,\bfell}(\partial/\partial \bfs)[h_{w_j}](w_j) 
\end{split} 
\end{equation}        
in the basis ${\mathscr B}_{\C[\bfz]/(\bfp(\bfz))}^{\prec\prec}$.     
\end{proof} 

\section{Two non-standard interpolation problems} 

Let $\bfp=(p_1,...,p_n)$ be a quasi-regular sequence in $\C[\bfs]$ and 
$$
{\mathscr B}_{\C[\bfs]/(\bfp)}^{\prec\prec} = \{\dot\bfs^{\bfbeta_k}\,;\, k=0,...,N(\bfp)-1\}
$$  
be a monomial basis of the $N(\bfp)$-dimensional 
$\C$-vector space 
$\C[\bfs]/(\bfp)$ which has previously been obtained thanks to the search for 
a Gr\"obner basis for $(\bfp)$ with respect to the prescribed ordering 
$\prec\prec$ on monomials in $\C[\bfs]$ (as recalled in section \ref{section1}). 
\vskip 2mm

We will denote as in section \ref{section1} as $\bfQ_{\bfp}[\mathscr B_{\C[\bfs]/(\bfp)}^{\prec\prec}]$ (see \eqref{matrixQp}) the matrix of the quadratic non-degenerated form \eqref{quadraticform} which stands as the residual generator of 
the module ${\rm Hom}_\C(\C[\bfs]/(\bfp),\C)$ (equipped with its structure of 
$\C[\bfs]/(\bfp)$-module as described in section \ref{section1}) constructed from the given quasi-regular sequence $(p_1,...,p_n)$ of generators of the polynomial zero dimensional ideal $(\bfp)$. 
The action of this generator is described as seen in section \ref{section1} by a   
standard list of N\oe therian operators 
\begin{equation}
{\rm Noeth}_\bfp[(\bfp)] := \Big\{\mathscr Q_{w_j,\bfell}(\partial/\partial \bfs)_{|w_j}\,;\, 
j=1,...,m,\ \bfell \in A_{w_j}\Big\} 
\end{equation} 
(see Definition \ref{noetherian}).
\vskip 2mm

We can now formulate (and indicate how to solve) the two following non-standard interpolation problems inspired by Problem 4.1 formulated in \cite{AJLV15} in the univariate case. 

\begin{problem}\label{problem1} 
Let $\bfp^{-1}(0)=\{w_1,...,w_m\}$ and $1\leq \mu \leq m$. Let $U$ be an open subset of $\C^n$ that contains $\{w_1,...,w_\mu\}$. Let   
$a_{j,\bfell}$ ($j=1,...,\mu$, $\bfell \in A_{w_j}$), together with $c$, be $(\sum_{j=1}^\mu \# A_{w_j}) +1$ given complex numbers. Describe the $\C$-affine space of functions $f~: U \rightarrow \C$ which are holomorphic in $U$ and moreover satisfy
\begin{equation}\label{eqproblem1} 
\sum\limits_{j=1}^\mu \sum\limits_{\bfell\in A_{w_j}} 
a_{j,\bfell}\, \mathscr Q_{w_j,\bfell}[f](w_j) = c. 
\end{equation}       
\end{problem} 

\begin{problem}\label{problem2} Let $w_1,...,w_\mu$ be $\mu$ distinct points 
in $\C^n$, together with $\mu$ elements $\bfnu_1$,...,$\bfnu_\mu$ in $(\N^*)^n$ (prescribed multi-vectors of multiplicities).  
Let $U$ be an open subset of $\C^n$ that contains $\{w_1,...,w_\mu\}$
Let $a_{j,\bfell}$ ($j=1...,\mu$, $\bfell \in \N^n$ such that 
$\bfell \prec \bfnu_j-\bfone$), together with $c$, be 
$\big(\sum_{\ell=1}^\mu \prod_{j=1}^n \nu_{\ell,j}\big)  + 1$ given complex numbers. 
Describe the $\C$-affine space of functions $f~: U \rightarrow \C$ which are holomorphic in $U$ and moreover satisfy
\begin{equation}\label{eqproblem2} 
\sum\limits_{j=1}^\mu \sum\limits_{\{\bfell\in \N^n\,;\, 
\bfell \prec \bfnu_j-\bfone\}} a_{j,\bfell}\, \Big(\frac{\partial^{|\bfell|}}{\partial \bfs^{\bfell}}\Big)[f](w_j) 
= c. 
\end{equation}       
\end{problem}     
   
\vskip 2mm

We start by indicating how to solve Problem \ref{problem1} in the particular case 
$\mu=m$. Recall that for each $j=1,...,m$, for each $\bfell \in A_{w_j}$, 
${\rm Supp}\, (\mathscr Q_{w_j,\bfell}) \subset A_{w_j}$, so that 
$$
\mathscr Q_{w_j,\bfell} (\partial/\partial \bfs) = 
\sum\limits_{\bflambda \in A_{w_j}} \tau_{j,\bfell,\bflambda} 
\, \frac{\partial^{|\bflambda|}}{\partial \bfs^{\bflambda}}  
$$
for some complex coefficients $\tau_{j,\bfell,\bflambda}$. 
  
\begin{lem}\label{lemma1}
Let $\{w_1,...,w_m\}=\bfp^{-1}(0)$, $U$ be 
an open subset of $\C^n$ containing $\bfp^{-1}(0)$ and the $a_{j,\bfell}$ ($j=1,...,m$, $\bfell \in A_{w_j}$), together with $c$, be complex numbers. 
For each $j=1,...,m$, let $h_{w_j}^{\bfa}$ be the germ in $\mathscr O_{\C^n,w_j}$ of 
$\bfs \longmapsto \sum_{\bflambda \in A_j} a_{j,\bflambda}\, (\bfs-w_j)^{\bflambda}/\bflambda!$  and $[\bfh_\bfw^{\bfa}]=[h_{w_1}^{\bfa},...,h_{w_m}^{\bfa}]$ .  
The following alternative holds~: 
\begin{itemize} 
\item either the coordinate system $\big(\alpha_0([\bfh_\bfw^\bfa]),...,\alpha_{N(\bfp)-1}([\bfh_\bfw^\bfa])\big)$ introduced in Proposition 
\ref{propWeil2} is the null system, which amounts to say that 
\begin{equation}\label{systemnull} 
\sum\limits_{\bflambda \in A_{w_j}} \tau_{j,\bfell,\bflambda}\, a_{j,\bflambda} = 0 
\quad \forall\, j=1,...,m,\ \forall\, \bfell \in A_{w_j},    
\end{equation} 
in which case the set of holomorphic functions $f~: U \rightarrow \C$ satisfying 
\eqref{eqproblem1} is empty when $c\not=0$ and is the whole space $H(U)$ when 
$c=0$~;    
\item either the coordinate system $\big(\alpha_0([\bfh_\bfw^\bfa]),...,\alpha_{N(\bfp)-1}([\bfh_\bfw^\bfa])\big)$ introduced in Proposition 
\ref{propWeil2} is non-zero, in which case a function $f\in H(U)$ satisfies \eqref{eqproblem1} if and only if 
\begin{equation}\label{systemnonnull}  
\begin{bmatrix} \alpha_{0}[f] & \dots & \alpha_{N(\bfp)-1)}[f]\end{bmatrix}\cdot 
\bfQ_{\bfp}[\mathscr B_{\C[\bfs]/(\bfp)}^{\prec\prec}]\, 
\begin{bmatrix}\alpha_0([\bfh_\bfw^{\bfa}]) \\ 
\vdots \\ 
\alpha_{N(\bfp)-1} ([\bfh_\bfw^\bfa])
\end{bmatrix} = c,   
\end{equation} 
which means that the vector $(\alpha_{0}[f],...,\alpha_{N(\bfp)-1)}[f])$ of the coordinates of $f$ in $H(U)/(\sum_1^n H(U)\, p_j)_{\rm loc}$  
lies in a specific affine hyperplane $\Pi^{\bfa}$ of $\C^N$ since 
the quadratic form \eqref{quadraticform} is non-degenerated.   
\end{itemize} 
\end{lem} 

\begin{proof} For each $j=1,...,m$, let 
$\psi_j\in \mathscr D(U,[0,1])$ with support in an arbitrary small 
neighborhood of $w_j$ (which does not contain any other zero of $\bfp$ and is such that the germ $h^\bfa_{w_j}$ admits a representant still denoted as $h^\bfa_{w_j}$ in it) such that furthermore   
$\psi_j\equiv 1$ about $w_j$. It follows from \eqref{reformulation} and 
\eqref{duality1} that one can rewrite the left-hand side of \eqref{eqproblem1} as 
\begin{multline*} 
\sum\limits_{j=1}^m 
\Big\langle \Big(\bigwedge_{j=1}^n \bar\partial(1/p_j)\Big)(\bfs)\,,\, f(\bfs)\, h_{w_j}^\bfa(\bfs)\, \psi_j(\bfs)\,  ds_1 \wedge \dots \wedge ds_n\Big\rangle \\ 
= \Big\langle \Big(\bigwedge_{j=1}^n \bar\partial(1/p_j)\Big)(\bfs)\,,\, {\rm Lag}[f](\bfs)\, {\rm Lag}[\bfh^{\bfa}_\bfw](\bfs)\, \Big(\sum_{j=1}^m \psi_j(\bfs)\Big)\,  ds_1 \wedge \dots \wedge ds_n\Big\rangle \\ 
= {\rm Res}\, \begin{bmatrix} 
{\rm Lag}[f](\bfs)\, {\rm Lag}[\bfh^{\bfa}_\bfw](\bfs)\, ds_1 \wedge \cdots \wedge ds_n \\ 
p_1(\bfs),...,p_n(\bfs)\end{bmatrix} \\ 
= \begin{bmatrix} \alpha_{0}[f] & \dots & \alpha_{N(\bfp)-1)}[f]\end{bmatrix}\cdot 
\bfQ_{\bfp}[\mathscr B_{\C[\bfs]/(\bfp)}^{\prec\prec}]\, 
\begin{bmatrix}\alpha_0([\bfh_\bfw^{\bfa}]) \\ 
\vdots \\ 
\alpha_{N(\bfp)-1} ([\bfh_\bfw^\bfa])
\end{bmatrix}. 
\end{multline*} 
We now conclude using the fact that the quadratic form $\bfQ_\bfp$ is non degenerated. The first situation in the alternative corresponds precisely (thanks to 
the role of the N\oe therian operators in the realisation of duality, see \eqref{duality2}) to the fact that the system of linear relations \eqref{systemnonnull} holds.    
\end{proof} 

\vskip 2mm
We may now state the solution to Problem \ref{problem1}. 

\begin{thm}\label{thmPb1} Let $\bfp^{-1}(0)=\{w_1,...,w_m\}$ and $1\leq \mu \leq m$. Let $U$ be an open subset of $\C^n$ that contains $\{w_1,...,w_\mu\}$. Let   
$a_{j,\bfell}$ ($j=1,...,\mu$, $\bfell \in A_{w_j}$), together with $c$, be $(\sum_{j=1}^\mu \# A_{w_j}) +1$ given complex numbers. For each $j=1,...,\mu$, let $h_{w_j}^{\bfa}$ be the germ in $\mathscr O_{\C^n,w_j}$ of 
$\bfs \longmapsto \sum_{\bflambda \in A_j} a_{j,\bflambda}\, (\bfs-w_j)^{\bflambda}/\bflambda!$. Let 
$$
[\bfh_\bfw^{\bfa}]= 
[h_{w_1}^\bfa,...,h_{w_\mu}^\bfa,0_{w_{\mu+1}},...,0_{w_n}]. 
$$ 
The following alternative then holds~: 
\begin{itemize} 
\item either the coordinate system $\big(\alpha_0([\bfh_\bfw^\bfa]),...,\alpha_{N(\bfp)-1}([\bfh_\bfw^\bfa])\big)$ introduced in Proposition 
\ref{propWeil2} is the null system, which amounts to say that 
\begin{equation}\label{systemnull1} 
\sum\limits_{\bflambda \in A_{w_j}} \tau_{j,\bfell,\bflambda}\, a_{j,\bflambda} = 0 
\quad \forall\, j=1,...,\mu,\ \forall\, \bfell \in A_{w_j},    
\end{equation} 
in which case the set of holomorphic functions $f~: U \rightarrow \C$ satisfying 
\eqref{eqproblem1} is empty when $c\not=0$ and is the whole space $H(U)$ when 
$c=0$~;    
\item either the coordinate system $\big(\alpha_0([\bfh_\bfw^\bfa]),...,\alpha_{N(\bfp)-1}([\bfh_\bfw^\bfa])\big)$ introduced in Proposition 
\ref{propWeil2} is non-zero, in which case a function $f\in H(U)$ satisfies \eqref{eqproblem1} if and only if \eqref{systemnonnull} holds,   
which means that the vector $(\alpha_{0}[f],...,\alpha_{N(\bfp)-1)}[f])$ of the coordinates of $f$ in $H(U)/(\sum_1^n H(U)\, p_j)_{\rm loc}$  
lies in a specific affine hyperplane $\Pi^{\bfa}$ of $\C^N$ since 
the quadratic form \eqref{quadraticform} is non-degenerated.    
\end{itemize} 
\end{thm} 

\begin{proof} One may assume that $\mu<n$ since the result is already proved when $\mu=m$ (Lemma \ref{lemma1}). 
\\
\noindent 
Let us first assume that $\partial U$ does not contain any of the points 
$w_j$ such that $\mu < j\leq m$. Let $\widetilde U$ be the union of $U$ with open balls 
$B(w_j,\varepsilon_j)$, $j=\mu+1,...,n$, where $\varepsilon_j$ is strictly smaller than the distance from $w_j$ to $\overline U$. Any holomorphic function 
$f~: U \rightarrow \C$ which satisfies \eqref{eqproblem2} can be considered as 
$\widetilde f_{|\widetilde U}$ where $\widetilde f~: \widetilde U \rightarrow \C$ satisfies \eqref{eqproblem1} with $\tilde a_{j,\bfell}=a_{j,\bfell}$ for 
any $j=1,...,\mu$ and any $\bfell\in A_{w_j}$, 
$\tilde a_{j,\bfell}=0$ for any index $j=\mu+1,...,m$ and any 
$\bfell\in A_{w_j}$, $\tilde c=c$. Conversely, given any such holomorphic solution  
$\tilde f~: \tilde U \rightarrow \C$ of \eqref{eqproblem1} in 
$\widetilde U$ (with the $a_{j,\bfell}$ replaced by $\tilde a_{j,\bfell}$), it restricts to $U$ as a solution of \eqref{eqproblem2}. The conclusion of Theorem 
\ref{thmPb1} follows then from that of Lemma \ref{lemma1}. 
\\
\noindent 
Consider now the case when $\partial U$ may contain some 
$w_j$ for $j=\mu+1,...,m$. Let $(U_k)_{k\geq 0}$ be an increasing sequence of open subsets such that $\{w_1,...,w_\mu\} \subset U_k\subset \overline{U_k} \subset U$ which exhausts $U$.         
It is equivalent to say that $f~: U \rightarrow \C$ is solution of \eqref{eqproblem1} in $U$ or that for any $k\in \N$, its restriction 
$f_{|U_k}$ is solution of \eqref{eqproblem1} in $U_k$ (the data $a_{w_j,\bfell}$ and $c$ remaining unchanged). For any $k\geq 0$, we showed that the alternative proposed 
in the statement of Theorem \ref{thmPb1} hold. As pointed out in Remark 
\ref{remProp1}, we also know that the coordinate system 
$(\alpha_0[f],...,\alpha_{N(\bfp)-1}[f])$ ($f$ being arbitrarily continued in 
$\widetilde U_k = U_k \cup \bigcup_{j=\mu+1}^m B(w_j,d(w_j,\overline{U_k})/2)$)
does depend only of the germs of $f$ at the points $w_j$ for $j=1,...,\mu$. 
On the other hand the condition on $[\bfh^\bfa_\bfw]$ which governs the alternative 
proposed in the statement of Theorem \ref{thmPb1} (when $U=U_k)$ does not depend on $k$. Hence this alternative still holds in the limit case $U_\infty=U$. Theorem \ref{thmPb1} is thus proved in general.         
\end{proof} 
\vskip 2mm 

Consider as an example the particular situation where $\bfp$ is a sequence of univariate monic polynomials respectively in the variables $s_1,...,s_n$, as in the series of Examples \ref{example1} till \ref{example4}. In this case, the $p_j$ being as in \eqref{univariate}, the 
set $A_w$ which is related to the point $w=(\xi_{1,\kappa_1},...,\xi_{n,\kappa_n})
\in \bfp^{-1}(0)$ is 
$$
A_w = \{\bfell\in \N^n\,;\, \bfell \prec \bfnu_w(\bfp) := (\nu_{1,\kappa_1},...,\nu_{n,\kappa_n})-\bf1\}, 
$$   
with cardinal $\prod_{\ell=1}^n \nu_{\ell,\kappa_\ell}$. Since one has in this case 
$$
\sum_{j=1}^m \# A_{w_j} = N(\bfp) = \dim_\C(\C[\bfs]/(\bfp)), 
$$
the first alternative in Lemma \ref{lemma1} leads in this particular case to 
$a_{w_j,\bflambda}=0$ for any $j=1,...,m$, for any $\bflambda \in A_{w_j}$ for dimension reasons. This holds also for which what concerns the first alternative in Theorem \ref{thmPb1} in the case $1\leq \mu<m$~: it boils down in this case to the conditions 
$a_{j,\bflambda}=0$ for any $j=1,...,\mu$, for any $\bfell\in \N^n$ such that 
$\bfell\prec \bfnu_w(\bfp)-\bfone$.  
\vskip 2mm

Consider now $\mu\geq 1$ distinct points in $\C^n$, paired with vectors of prescribed multiplicities $\bfnu_1,\cdots,\bfnu_\mu$ in $(\N^*)^n$ as in Problem \ref{problem2}. If $w_j=(\xi_{j,1},...,\xi_{j,n})$, let us form the $n$ univariate monic polynomials 
$$
p_j(s_j) := \prod\limits_{\ell = 1}^\mu (s_j - \xi_{\ell,j})^{\nu_{\ell,j}}\in \C[\bfs_j],\quad j=1,...,n.  
$$
Let $\bfd =(\deg p_1,...,\deg p_n)$, where 
$\deg p_j=\sum_{\ell=1}^\mu \nu_{\ell,j}$ for $j=1,...,n$. A monomial basis 
for $\C[\bfs]/(\bfp)$ is provided thanks to Euclid's algorithm in the separated variables $s_1,...,s_n$ as 
$$
\mathscr B_{\bfp}^{\rm euclid} = \{\dot \bfs^{\bfk}\,;\, \bfk\in \N^n\ {\rm with}\ \bfk \prec \bfd-\bf1\}. 
$$
From now on, one organizes this basis with respect to the lexicographical order on the multi-exposants of monomials in $\C[\bfs]$. Keeping to such ordering, let 
\begin{equation}\label{matrixfqpart} 
\bfQ_\bfp^{\rm euclid} := \Bigg[ {\rm Res}\, \begin{bmatrix} 
\bfs^{\bfk_1 + \bfk_2}\, ds_1 \wedge \cdots \wedge ds_n\\ 
p_1(s_1),...,p_n(s_n) \end{bmatrix}\Bigg]_{
\stackrel{\bfk_1,\bfk_2\in \N^n}
{\bfk_1,\bfk_2 \prec \bfd-\bf1}}. 
\end{equation}  
The $(\bfk_1,\bfk_2)$ entry of such a matrix is obtained as the coefficient 
$\gamma_{\bfk_1,\bfk_2,\bfd-\bf1}$ of 
$\bfs^{\bfd -\bf1}$ in the expansion as a geometric series about $\bf0$ of 
$$
\bfs \mapsto \bfs^{\bfk_1 + \bfk_2}\, \prod\limits_{j=1}^n 
\Big(\frac{1} 
{1+\frac{p_j(s_j)-s_j^{d_j}}{s_j^{d_j}}}\Big) = \sum\limits_{\{\bfkappa \in\Z^n\,;\, \kappa\prec \bfk_1+\bfk_2\}} 
\gamma_{\bfk_1,\bfk_2,\bfkappa}\, \bfs^{\bfkappa}. 
$$ 
For each $j=1,...,\mu$, let $\psi_j\in \mathscr D(\C^n,[0,1])$ be a test-function which is identically equal to $1$ near $w_j$ and identically equal to $0$ about 
any point in $\bfp^{-1}(0)\setminus \{w_j\}$. Given a list $\bfa= 
\{a_{j,\bfell}\,;\, j=1,...,\mu,\ \bfell\in \N^n \ {\rm with}\  \bfell \prec \bfnu_j-\bf1\}$, consider the Lagrange interpolator ${\rm Lag}[\bfh_{\bfw}^{\bfa}]\in \C[\bfz]$ defined as 
\begin{multline}\label{laginterpfinal}  
{\rm Lag}[\bfh_{\bfw}^{\bfa}](\bfz) 
 = \sum\limits_{j=1}^\mu 
\Big\langle 
\Big(\bigwedge\limits_{j=1}^n 
\bar\partial(1/p_j)\Big)(\bfs)\,,\\  
\Big(\sum\limits_{\bfell\prec 
\bfnu_j-\bf1} 
a_{j,\bfell} 
\, \frac{(\bfs-w_j)^{\bfell}}{\bfell!}\Big)\, 
\Big(\prod\limits_{j=1}^n 
\frac{p_j(z_j) - p_j(s_j)}{z_j-s_j}\Big)\, 
\psi_j(\bfs)\, ds_1 \wedge \cdots \wedge ds_n\Big\rangle \\ 
= \sum\limits_{\{\bfk\in \N^n\,;\, \bfk \prec \bfd-\bfone\}} 
\tau_{\bfk} (\bfa)\, \bfz^\bfk.       
\end{multline}     

One can now state the following result with respect to Problem \ref{problem2}. 

\begin{thm} Let $w_1,...,w_\mu$ be $\mu$ distinct points 
in $\C^n$, together with $\mu$ elements $\bfnu_1$,...,$\bfnu_\mu$ in $(\N^*)^n$. 
Let $U$ be an open subset of $\C^n$ that contains $\{w_1,...,w_\mu\}$. 
Let $a_{j,\bfell}$ ($j=1...,\mu$, $\bfell \in \N^n$ such that 
$\bfell \prec \bfnu_j-\bfone$), together with $c$, be 
$\big(\sum_{\ell=1}^\mu \prod_{j=1}^n \nu_{\ell,j}\big)  + 1$ given complex numbers 
such that the $a_{j,\bfell}$ (for $j=1,...,\mu$, $\bfell\prec \bfnu_j$) are not all equal to $0$. An holomorphic function $f~: U \rightarrow \C$ satisfies 
\eqref{eqproblem2} if and only if there are coefficients $\alpha_{\bfk}$, 
$\bfk \prec \bfd$, where $d_j = \sum_{\ell=1}^\mu \nu_{\ell,j}$ for $j=1,...,\mu$,  
such that 
$$
f(s) = \sum\limits_{\{\bfk\in \N^n\,;\, \bfk\prec \bfd-\bfone\}} \alpha_\bfk\, \bfs^{\bfk} + g(\bfs)  
$$
where $g$ is an holomorphic function in $U$ which belongs locally to the ideal 
generated by the univariate polynomials $p_j(s_j)=\sum_{\ell=1}^\mu 
(s_j-w_{\ell,j})^{\nu_{\ell,j}}$ ($j=1,...,n$) and 
$(\alpha_\bfk)_{\bfk\prec \bfd-\bfone}$ satisfies 
$$
\begin{bmatrix} \alpha_{\bfzero} & \cdots & \alpha_{\bfd-\bf1}\end{bmatrix}
\cdot \bfQ_\bfp^{\rm  euclid} \cdot \begin{bmatrix} \tau_{\bf0}(\bfa) \\ 
\vdots \\ \tau_{\bfd-\bf1}(\bfa)\end{bmatrix} = c,    
$$
where the coefficients $\tau_{\bfk}(\bfa)$ ($\bfk\in \N^n$, $\bfk\prec 
\bfd-\bfone$) are those of the Lagrange interpolator \eqref{laginterpfinal}.
\end{thm} 

\begin{proof} This is an immediate application of Theorem \ref{thmPb1}. 
\end{proof} 
\section{Conclusions}
We have presented a new algebraic approach, based on residue theory and duality (see \cite{MR1704477,MR1421336,ElkMou})
to solve a scalar interpolation problem in several complex variable. In a future work we plan to exploit the present methods in the matricial case to 
study the counterpart of the bitangential interpolation problem 
(see e.g. \cite{bgr}) in the present setting. We 
focused here on the algebraic point of view. Hilbert space constraints will be considered elswhere. Both the methods and results
are different from the ones classically related up to now to interpolation in the 
Drury-Arveson space or Schur multipliers and Schur-Agler classes
(see \cite {agler-hellinger,akap1,MR2069781,MR3350215} for the latter).
\bibliographystyle{plain}

\def\cprime{$'$} \def\cprime{$'$} \def\cprime{$'$}
  \def\lfhook#1{\setbox0=\hbox{#1}{\ooalign{\hidewidth
  \lower1.5ex\hbox{'}\hidewidth\crcr\unhbox0}}} \def\cprime{$'$}
  \def\cprime{$'$} \def\cprime{$'$} \def\cprime{$'$} \def\cprime{$'$}
  \def\cprime{$'$}


\begin{thebibliography}{10}

\bibitem{agler-hellinger}
J.~Agler.
\newblock {\em On the representation of certain holomorphic functions defined
  on a polydisk}, volume~48 of {\em {Operator {T}heory: {A}dvances and
  {A}pplications}}, pages 47--66.
\newblock Birkh{\" a}user Verlag, Basel, 1990.

\bibitem{AY}
I.A. Aizenberg and A.~P. Yuzhakov.
\newblock {\em Integral representations and residues in multidimensional
  complex analysis}, volume~58 of {\em Translations of Mathematical
  Monographs}.
\newblock American Mathematical Society, Providence, RI, 1983.
\newblock Translated from the Russian by H. H. McFaden, Translation edited by
  Lev J. Leifman.

\bibitem{AJLV15}
D.~Alpay, P.~Jorgensen, I.~Lewkowicz, and D.~Volok.
\newblock A new realization of rational functions, with applications to linear
  combination interpolation, the {C}untz relations and kernel decompositions.
\newblock {\em Complex Var. Elliptic Equ.}, 61(1):42--54, 2016.

\bibitem{akap1}
D.~Alpay and H.T. Kaptano\u{g}lu.
\newblock Some finite-dimensional backward shift-invariant subspaces in the
  ball and a related interpolation problem.
\newblock {\em Integral Equation and Operator Theory}, 42:1--21, 2002.

\bibitem{MR2069781}
J.A. Ball and V.~Bolotnikov.
\newblock Realization and interpolation for {S}chur-{A}gler-class functions on
  domains with matrix polynomial defining function in ${C}^n$.
\newblock {\em J. Funct. Anal.}, 213(1):45--87, 2004.

\bibitem{bgr}
J.A. Ball, I.~Gohberg, and L.~Rodman.
\newblock {\em Interpolation of rational matrix functions}, volume~45 of {\em
  Operator {T}heory: {A}dvances and {A}pplications}.
\newblock Birkh{\" a}user Verlag, Basel, 1990.

\bibitem{MR3350215}
J.A. Ball and D.S. Kaliuzhnyi-Verbovetskyi.
\newblock Schur-{A}gler and {H}erglotz-{A}gler classes of functions:
  positive-kernel decompositions and transfer-function realizations.
\newblock {\em Adv. Math.}, 280:121--187, 2015.

\bibitem{BGVY}
C.A. Berenstein, R.~Gay, A.~Vidras, and A.~Yger.
\newblock {\em Residue currents and {B}ezout identities}, volume 114 of {\em
  Progress in Mathematics}.
\newblock Birkh\"auser Verlag, Basel, 1993.

\bibitem{MR1704477}
C.A. Berenstein and A.~Yger.
\newblock Residue calculus and effective {N}ullstellensatz.
\newblock {\em Amer. J. Math.}, 121(4):723--796, 1999.

\bibitem{BoH}
J.-Y. Boyer and M.~Hickel.
\newblock Extension dans un cadre alg\'ebrique d'une formule de {W}eil.
\newblock {\em Manuscripta Math.}, 98(2):195--223, 1999.

\bibitem{MR1421336}
J.~P. Cardinal and B.~Mourrain.
\newblock Algebraic approach of residues and applications.
\newblock In {\em The mathematics of numerical analysis ({P}ark {C}ity, {UT},
  1995)}, volume~32 of {\em Lectures in Appl. Math.}, pages 189--210. Amer.
  Math. Soc., Providence, RI, 1996.

\bibitem{ElkMou}
M.~Elkadi and B.~Mourrain.
\newblock {\em Introduction \`a la r\'esolution des syst\`emes polynomiaux},
  volume~59 of {\em Math\'ematiques \& Applications (Berlin) [Mathematics \&
  Applications]}.
\newblock Springer, Berlin, 2007.

\bibitem{Glea}
A.~Gleason.
\newblock The {C}auchy-{W}eil theorem.
\newblock {\em J. Math. Mech.}, 12:429--444, 1963.

\bibitem{Jel}
Z.~Jelonek.
\newblock On the effective {N}ullstellensatz.
\newblock {\em Invent. Math.}, 162(1):1--17, 2005.

\bibitem{Pl}
A.~P{\l}oski.
\newblock On the {N}oether exponent.
\newblock {\em Bull. Soc. Sci. Lett. L\'odz}, 40(1-10):23--29 (1991), 1990.

\bibitem{Sam}
H. Samuelsson.
\newblock Analytic continuation of residue currents.
\newblock {\em Ark. Mat.}, 47(1):127--141, 2009.

\bibitem{Tsi}
A.~Tsikh.
\newblock {\em Multidimensional residues and their applications}, volume 103 of
  {\em Translations of Mathematical Monographs}.
\newblock American Mathematical Society, Providence, RI, 1992.
\newblock Translated from the 1988 Russian original by E. J. F. Primrose.

\bibitem{TsiY}
A.~Tsikh and A.~Yger.
\newblock Residue currents.
\newblock {\em J. Math. Sci. (N. Y.)}, 120(6):1916--1971, 2004.
\newblock Complex analysis.

\bibitem{Weil}
A.~Weil.
\newblock {\em L'int\'egration dans les groupes topologiques et ses
  applications}.
\newblock Actual. Sci. Ind., no. 869. Hermann et Cie., Paris, 1940.
\newblock [This book has been republished by the author at Princeton, N. J.,
  1941.].

\end{thebibliography}
\end{document}